\documentclass[12pt,reqno]{amsart}

\usepackage{amssymb}

\newtheorem{theorem}{Theorem}[section]
\newtheorem{proposition}[theorem]{Proposition}
\newtheorem{lemma}[theorem]{Lemma}

\theoremstyle{definition}
\newtheorem{remark}[theorem]{Remark}

\numberwithin{equation}{section}

\newcommand{\C}{\mathbb C}

\newcommand{\Z}{\mathbb Z}

\begin{document}

\baselineskip=15.5pt

\title[Complex Lagrangians in a hyperK\"ahler manifold and Albanese]{Complex Lagrangians in a
hyperK\"ahler manifold and the relative Albanese}

\author[I. Biswas]{Indranil Biswas}

\address{School of Mathematics, Tata Institute of Fundamental
Research, Homi Bhabha Road, Mumbai 400005, India}

\email{indranil@math.tifr.res.in}

\author[T. G\'omez]{Tom\'as L. G\'omez}

\address{Instituto de Ciencias Matem\'aticas (CSIC-UAM-UC3M-UCM),
Nicol\'as Cabrera 15, Campus Cantoblanco UAM, 28049 Madrid, Spain}

\email{tomas.gomez@icmat.es}

\author[A. Oliveira]{Andr\'e Oliveira}

\address{Centro de Matem\'atica da Universidade do Porto (CMUP), Faculdade de Ci\^encias,
Rua do Campo Alegre, 687, 4169-007 Porto, Portugal
\newline\indent \textsl{On leave from:}\newline\indent Departamento de Matem\'atica, Universidade
de Tr\'as-os-Montes e Alto Douro, UTAD,\newline
Quinta dos Prados, 5000-911 Vila Real, Portugal}
\email{andre.oliveira@fc.up.pt / agoliv@utad.pt}



\date{}

\begin{abstract}
Let $M$ be the moduli space of complex Lagrangian submanifolds of a
hyperK\"ahler manifold $X$, and let 
$\varpi\, :\, \widehat{\mathcal{A}}\, \longrightarrow\, M$ be the relative Albanese over $M$.
We prove that
$\widehat{\mathcal{A}}$ has a natural holomorphic symplectic structure.
The projection $\varpi$ defines a
completely integrable structure on the symplectic
manifold $\widehat{\mathcal{A}}$. In particular, the
fibers of $\varpi$ are complex Lagrangians with respect to the symplectic form
on $\widehat{\mathcal{A}}$. We also prove analogous results for the relative Picard over $M$.
\end{abstract}

\maketitle

\section{Introduction}

A compact K\"ahler manifold admits a holomorphic symplectic form if and only if it
admits a hyperK\"ahler structure \cite{Be}, \cite{Ya}. To explain this, let $X$ be a compact manifold
equipped with almost complex structures $J_1,\, J_2,\, J_3$, and let $g$ be a Riemannian metric on $X$,
such that $(X,\, J_1,\, J_2,\, J_3,\, g)$ is a hyperK\"ahler manifold. Then $g$ defines a $C^\infty$ isomorphism,
$$
T^{0,1}X\, \stackrel{g_1}{\longrightarrow}\,(T^{1,0}X)^*\, ,
$$
where $T^{\mathbb R}X\otimes {\mathbb C}\,=\, T^{1,0}X\oplus T^{0,1}X$ is the type decomposition with
respect to the almost complex structure $J_1$; also $J_2$ produces a $C^\infty$ isomorphism
$$
T^{1,0}X\, \stackrel{J'_2}{\longrightarrow}\,T^{0,1}X\, .
$$
The composition of homomorphisms
$$
T^{1,0}X\, \stackrel{J'_2}{\longrightarrow}\,T^{0,1}X \, \stackrel{g_1}{\longrightarrow}\,(T^{1,0}X)^*\, ,
$$
which is a section of $(T^{1,0}X)^*\otimes (T^{1,0}X)^*$, is actually is a holomorphic symplectic form on the
compact K\"ahler manifold $(X,\, J_1,\, g)$. The compact K\"ahler manifold $(X,\, J_1,\, g)$ is Ricci--flat.
Conversely, if a compact K\"ahler manifold admits a
holomorphic symplectic form, then its canonical line bundle is holomorphically trivial and hence it admits
a Ricci--flat K\"ahler metric \cite{Ya}. Let $(X,\, J_1 ,\, g)$ be a Ricci--flat compact K\"ahler
manifold equipped with a holomorphic symplectic form. Then we may recover $J_2$ by reversing the above
construction. Finally, we have $J_3\,=\, J_1\circ J_2$.

Let $X$ be a compact K\"ahler Ricci--flat manifold admitting a holomorphic
symplectic form. Fix a Ricci--flat K\"ahler form $\omega$ on $X$ (such
a K\"ahler form exists \cite{Ya}), and take a holomorphic symplectic
form $\Phi$ on $X$. It is known that $\Phi$ is parallel (meaning, covariant
constant) with respect to the Levi--Civita connection on $X$
associated to $\omega$ \cite[p.~760, ``Principe~de~Bochner'']{Be},
\cite[p.~142]{YB}.

Let $M$ denote the moduli space of compact complex submanifolds of $X$
that are Lagrangian with respect to the symplectic form $\Phi$. Consider the
corresponding universal family of Lagrangians
\begin{equation}\label{ez}
{\mathcal Z}\, \longrightarrow\, M\, .
\end{equation}
Let
$$
\varpi\, :\, \widehat{\mathcal{A}}\, \longrightarrow\, M
$$
be the relative Albanese over $M$. So for any Lagrangian $L\, \in\, M$,
the fiber of $\widehat{\mathcal{A}}$ over $L$ is $\text{Alb}(L)\,=\,
H^0(L,\,\Omega^1_L)^*/H_1\,(L,\,\Z)$. This $\varphi$ is a holomorphic family
of compact complex tori over $M$.

We prove the following (see Theorem \ref{theo1}):

\textit{The complex manifold $\widehat{\mathcal{A}}$ has a natural holomorphic
symplectic form.}

The symplectic form on $\widehat{\mathcal{A}}$ is constructed using the
canonical Liouville symplectic form on the holomorphic cotangent bundle
$T^*M$ of $M$. The symplectic form $\Phi$ on $X$ is implicitly used in the
construction of the symplectic form on $\widehat{\mathcal{A}}$. Recall that the
Lagrangian submanifolds, and hence $M$, are defined using $\Phi$.

We prove the following (see Lemma \ref{lem-1}):

\textit{The projection $\varpi\, :\, \widehat{\mathcal{A}}\, \longrightarrow\, M$ defines a
completely integrable structure on the symplectic
manifold $\widehat{\mathcal{A}}$. In particular, the
fibers of $\varpi$ are Lagrangians with respect to the
symplectic form on $\widehat{\mathcal{A}}$.}

In Section \ref{se4} we consider the relative Picard bundle over $M$ for the 
family ${\mathcal Z}$ in \eqref{ez}. Let
$$
\varpi_0\, :\, \mathcal{A}\, \longrightarrow\, M
$$
be the relative Picard bundle for the family ${\mathcal Z}$.

We prove that $\mathcal{A}$ is equipped with a natural holomorphic symplectic structure;
see Proposition \ref{proposi1}.

Let $\Theta_{\mathcal A}$ denote the above mentioned holomorphic symplectic structure
on $\mathcal{A}$. The following lemma is proved (see Lemma \ref{lem2}):

\textit{The projection $\varpi_0\,:\, \mathcal{A}\, \longrightarrow\, M$ defines a
completely integrable structure on $\mathcal{A}$ for the symplectic
form $\Theta_{\mathcal A}$.}

These results are natural generalizations of some known cases of integrable systems, such as the
Hitchin system \cite{Hi1} or the Mukai system \cite{Mu} (see also \cite{DEL}). 
One of our motivations has been mirror symmetry; we hope to come back to the study of
$\mathcal{A}$ and $\widehat{\mathcal{A}}$ from the point of view of
hyperK\"ahler geometry.

\section{Cotangent bundle of family of Lagrangians}\label{sect2}

Let $X$ be a compact K\"ahler Ricci--flat manifold of complex dimension $2d$ equipped with a K\"ahler
form $\omega$. Let $\Phi$ be a holomorphic symplectic form on $X$ which is parallel
with respect to the Levi--Civita connection on $X$ given by the K\"ahler metric on $X$
associated to $\omega$.

A complex Lagrangian submanifold of $X$ is a compact complex submanifold $L\,\subset\, X$
of complex dimension $d$ such that $\iota^*\Phi\,=\,0$, where
\begin{equation}\label{e1}
\iota\, :\, L\, \hookrightarrow\, X
\end{equation}
is the inclusion map.

It is known that the infinitesimal deformations of a complex Lagrangian submanifold
of $X$ are unobstructed \cite{Mc}, \cite{Vo}, \cite{Hi2}. Furthermore, the moduli space
of complex Lagrangian submanifolds of $X$ is a special K\"ahler manifold \cite[p. 84,
Theorem 3]{Hi2}.

Let $M$ be the moduli space of complex Lagrangian submanifolds of $X$. Let
$$
L\,\subset\, X
$$
be a complex Lagrangian submanifold. The point of $M$ representing $L$ will also be
denoted by $L$. Let $N_L\, \longrightarrow\, L$ be the normal bundle
of $L\, \subset\, X$; it is a quotient bundle of $\iota^*TX$ of rank $d$, where $\iota$ is
the map in \eqref{e1}.
The infinitesimal deformations of the complex submanifold
$L$ are parametrized by $H^0(L,\, N_L)$. Since
$L$ is complex Lagrangian, the holomorphic symplectic form $\Phi$ on $X$ produces a holomorphic isomorphism
$$
N_L\,\stackrel{\sim}{\longrightarrow}\, (TL)^*\,=\, \Omega^1_L\, ,
$$
where $TL$ (respectively, $\Omega^1_L$) is the holomorphic tangent (respectively, cotangent)
bundle of $L$. Using this isomorphism we have
$$
H^0(L,\, N_L)\,=\, H^0(L,\, \Omega^1_L)\, .
$$
Among the infinitesimal deformations of the complex submanifold
$L$, there are those which arise from deformations within the category of complex Lagrangian submanifolds,
meaning those arise from deformations of complex Lagrangian submanifolds as Lagrangian submanifolds.
An infinitesimal deformation $$\alpha\, \in\, H^0(L,\, N_L)\,=\, H^0(L,\, \Omega^1_L)$$ of $L$ lies
in this subclass if and only if the holomorphic $1$--form $\alpha$
on $L$ is closed \cite[pp.~78--79]{Hi2}. But any holomorphic $1$--form on $L$ is closed because $L$ is K\"ahler.
Therefore, $M$ is in fact an open subset of the corresponding Douady space for $X$.

Consider the Ricci--flat K\"ahler form $\omega$ on $X$. The form
\begin{equation}\label{wl}
\omega_L\, :=\, \iota^*\omega
\end{equation}
on $L$ is also K\"ahler, where $\iota$ is the map in \eqref{e1}.
Therefore, the pairing
\begin{equation}\label{e2}
\phi_L\, :\, H^0(L,\, \Omega^1_L)\otimes H^1(L,\, {\mathcal O}_L)\, \longrightarrow\,
\mathbb C\, , \ \ w\otimes c\, \longmapsto\, \int_L w\wedge c\wedge\omega^{d-1}_L
\end{equation}
is nondegenerate. We shall identify $H^1(L,\, {\mathcal O}_L)$ with $H^0(L,\, \Omega^1_L)^*$
using this nondegenerate pairing.

It was noted above that
\begin{equation}\label{n1}
T_LM\,=\, H^0(L,\, N_L)\,=\, H^0(L,\, \Omega^1_L)\, .
\end{equation}
Using the pairing in \eqref{e2} and \eqref{n1}, we
have
\begin{equation}\label{e3}
T^*_LM\,=\, H^0(L,\, \Omega^1_L)^*\,=\, H^1(L,\, {\mathcal O}_L)\, .
\end{equation}

Let
\begin{equation}\label{e-1}
{\mathcal Z}\, \subset\, X\times M
\end{equation}
be the universal family of complex Lagrangians over $M$. So ${\mathcal Z}$ is the locus of
all $(x,\, L')\, \in\, X\times M$ such that $x\, \in\, L'
\, \subset\, X$. Consider the natural
projections $p_X\, :\, X\times M\, \longrightarrow\, X$ and $p_M\, :\, X\times M\, \longrightarrow\, M$. Let
\begin{equation}\label{pq}
p\, :\, {\mathcal Z}\, \longrightarrow\, X \ \ \text{ and }\ \
q\, :\, {\mathcal Z}\, \longrightarrow\, M
\end{equation}
be the restrictions of $p_X$ and $p_M$ respectively to the submanifold ${\mathcal Z}\, \subset\, X\times M$.
So $p(q^{-1}(L'))\, \subset\, X$ for every $L'\, \in\, M$ is
the Lagrangian $L'$ itself.

Let $\Omega^1_{{\mathcal Z}/M}\, \longrightarrow\, {\mathcal Z}$ be the relative cotangent bundle
for the projection $q$ to $M$ in \eqref{pq}. It fits in the short exact sequence of holomorphic vector bundles
$$
0 \, \longrightarrow\, q^*\Omega^1_M \, \longrightarrow\, \Omega^1_{\mathcal Z}
\, \longrightarrow\, \Omega^1_{{\mathcal Z}/M}\, \longrightarrow\, 0
$$
over $M$.
The direct image $R^1q_*{\mathcal O}_{\mathcal Z}$ fits in the short exact sequence of sheaves on $M$
\begin{equation}\label{a1}
0\, \longrightarrow\,R^0q_*\Omega^1_{{\mathcal Z}/M}\, \longrightarrow\, R^1q_*\underline{\mathbb C}
\, \longrightarrow\, R^1q_*{\mathcal O}_{\mathcal Z}\, \longrightarrow\, 0\, ,
\end{equation}
where $\underline{\mathbb C}$ is the constant sheaf on ${\mathcal Z}$ with stalk $\mathbb C$. The direct image
$R^1q_*\underline{\mathbb C}$ is a flat complex vector bundle, equipped with the Gauss--Manin connection.
We briefly recall the construction of the Gauss--Manin connection. For any point $L'\,=\, y \, \in\, M$, let
$U_y\, \subset\, M$ be a contractible open neighborhood of $y$. Since $U_y$ is contractible,
the inverse image $q^{-1}(U_y)$ is diffeomorphic to $U_y\times L'$ such that the diffeomorphism
between $q^{-1}(U_y)$ and $U_y\times L'$ takes $q$ to the natural projection from $U_y\times L'$ to $U_y$.
Using this diffeomorphism, the restriction of $R^1q_*\underline{\mathbb C}$ to $U_y$ coincides with the
trivial vector bundle
\begin{equation}\label{tvb}
U_y\times H^1(L',\, {\mathbb C}) \, \longrightarrow\, U_y
\end{equation}
with fiber $H^1(L',\, {\mathbb C})$. Using this isomorphism between $(R^1q_*\underline{\mathbb C})\vert_{U_y}$
and the trivial vector bundle $U_y\times H^1(L',\, {\mathbb C})$ in \eqref{tvb}, the trivial connection on
the trivial vector bundle in \eqref{tvb} produces a flat connection on $(R^1q_*\underline{\mathbb C})\vert_{U_y}$.
This connection on $(R^1q_*\underline{\mathbb C})\vert_{U_y}$ does not depend on the choice of the
diffeomorphism between $q^{-1}(U_y)$ and $U_y\times L'$. Consequently, these locally defined flat connections
on $R^1q_*\underline{\mathbb C}$ patch together compatibly to define a flat connection on
$R^1q_*\underline{\mathbb C}$. This flat connection is the Gauss--Manin connection mentioned above.

Since the Gauss--Manin connection on
$R^1q_*\underline{\mathbb C}$ is flat, and $M$ is a complex manifold, the Gauss--Manin connection
produces a natural holomorphic structure
on the $C^\infty$ vector bundle $R^1q_*\underline{\mathbb C}$. The direct image
$R^0q_*\Omega^1_{{\mathcal Z}/M}$ is a holomorphic subbundle of $R^1q_*\underline{\mathbb C}$,
but it is not preserved by the flat connection in general. The holomorphic structure
on the quotient $R^1q_*{\mathcal O}_{\mathcal Z}$ induced by that of $R^1q_*\underline{\mathbb C}$ coincides with 
its own holomorphic structure; the fiber of $R^1q_*{\mathcal O}_{\mathcal Z}$
over any $L'\, \in\, M$ is $H^1(L',\, {\mathcal O}_{L'})$.

Since $\omega_L$ in \eqref{wl} is the restriction of a global K\"ahler form on $X$, the section of
$R^2q_*\underline{\mathbb C}$
$$
M\,\longrightarrow\, R^2q_*\underline{\mathbb C}\, ,\ \ L' \,\longmapsto\, [\omega_{L'}]\,=\,
[\omega\vert_{L'}]\,\in\, H^2(L', \, {\mathbb C}) 
$$
is covariant constant with respect to the Gauss--Manin connection on $R^2q_*\underline{\mathbb C}$.
Consequently, the homomorphism
\begin{equation}\label{ep}
(R^1 q_*\underline{\mathbb C})\otimes (R^1 q_*\underline{\mathbb C})\, \longrightarrow\,
\underline{\mathbb C}\, ,
\end{equation}
that sends any $v\otimes w \, \in\,
(R^1 q_*\underline{\mathbb C})_t\otimes (R^1 q_*\underline{\mathbb C})_t$, $t\, \in\, M$,
to
$$
\int_{q^{-1}(t)} v\wedge w\wedge (\omega\vert_{q^{-1}(t)})^{d-1}\, \in\, \mathbb C
$$
is also covariant constant with respect to the connection on $(R^1 q_*\underline{\mathbb C})
\otimes (R^1 q_*\underline{\mathbb C})$ induced by the Gauss--Manin connection
on $R^1 q_*\underline{\mathbb C}$. Hence the pairing
in \eqref{ep} produces a holomorphic isomorphism of vector bundles
\begin{equation}\label{ep1}
(q_*\Omega^1_{{\mathcal Z}/M})^* \,\stackrel{\sim}{\longrightarrow} \, R^1q_*{\mathcal O}_{\mathcal Z}
\end{equation}
on $M$. The restriction of this isomorphism to any point $L\, \in\, M$ coincides
with the isomorphism $H^0(L,\, \Omega^1_L)^*\,=\, H^1(L,\, {\mathcal O}_L)$ in \eqref{e3}.

On the other hand, the pointwise isomorphisms in \eqref{n1} combine together
to produce a holomorphic isomorphism of vector bundles
\begin{equation}\label{ep2}
\Omega^1_M \,\stackrel{\sim}{\longrightarrow} \, (q_*\Omega^1_{{\mathcal Z}/M})^*
\end{equation}
on $M$. Composing the isomorphisms in \eqref{ep1} and \eqref{ep2}, we obtain
a holomorphic isomorphism of vector bundles
\begin{equation}\label{e4}
\chi\, :\, \Omega^1_M \,\longrightarrow \, R^1q_*{\mathcal O}_{\mathcal Z}
\end{equation}
over $M$.

For notational convenience, the total space of $R^1q_*{\mathcal O}_{\mathcal Z}$ will be
denoted by ${\mathcal Y}$. Let
\begin{equation}\label{e6}
\gamma\, :\, {\mathcal Y}\,\longrightarrow \, M
\end{equation}
be the natural projection. Consider the canonical Liouville $1$-form on $\Omega^1_M$.
Using the isomorphism $\chi$ in \eqref{e4} this Liouville $1$-form on $\Omega^1_M$ gives
a holomorphic $1$-form on ${\mathcal Y}$. Let
\begin{equation}\label{e5}
\theta\,\in\, H^0({\mathcal Y}, \, \Omega^1_{\mathcal Y})
\end{equation}
be the holomorphic $1$-form given by the Liouville $1$-form on $\Omega^1_M$. Note that
for any $L\,\in\, M$, and any $v\,\in\, \gamma^{-1}(L)$, the form
$$
\theta(v)\, :\, T_v{\mathcal Y}\,\longrightarrow \, \mathbb C
$$
coincides with the composition of homomorphisms
$$
T_v{\mathcal Y}\,\stackrel{d\gamma}{\longrightarrow}\, T_L M\,=\, H^0(L,\, \Omega^1_L)
\,\stackrel{\phi_L(-,v)}{\longrightarrow}\,\mathbb C\, ,
$$
where $\phi_L$ is the bilinear pairing constructed in \eqref{e2}, and $d\gamma$ is the
differential of the projection $\gamma$ in \eqref{e6}; the above identification $$T_L M\,=\,
H^0(L,\, \Omega^1_L)$$ is the one constructed in \eqref{n1}.

It is straight-forward to check that the $2$-form
\begin{equation}\label{e10}
d\theta\, \in\, H^0({\mathcal Y},\, \Omega^2_{\mathcal Y})
\end{equation}
is a holomorphic symplectic form on the manifold
${\mathcal Y}$ in \eqref{e6}. Indeed, $d\theta$ evidently coincides with
the $2$-form on ${\mathcal Y}$ given by the Liouville symplectic form on
$\Omega^1_M$ via the isomorphism $\chi$ in \eqref{e4}.

\section{The family of Albanese tori}

Take a compact complex Lagrangian submanifold $L\subset X$ represented by a point of $M$.
We know that $T_L^*M\,\cong\, H^0(L,\,\Omega^1_L)^*$
(see \eqref{n1}). Note that the non-degenerate pairing
$$
H^{1,0}(L)\otimes H^{d-1,d}(L)\, \longrightarrow\,
\mathbb C\, , \ \ \alpha\otimes \beta\, \longmapsto\, \int_L \alpha\wedge \beta\, ,
$$
which is also the Serre duality pairing, yields an isomorphism
$$H^0(L,\,\Omega^1_L)^*\,\cong\, H^d(L,\,\Omega^{d-1}_L)\, .$$

For a fixed $L$, we have the Hodge decomposition $$H^{2d-1}(L,\,\C)\,=\,H^{d,d-1}(L)\oplus H^{d-1,d}(L)\, ;$$
but if we move $L$ in the family $M$, meaning if we consider the universal family $$q\,:\,\mathcal{Z}
\, \longrightarrow\, M$$ in \eqref{e-1},
then only $R^{d-1}q_*\Omega^d_{{\mathcal Z}/M}$ is a holomorphic subbundle of $R^{2d-1}q_*\underline{\C}$,
and we have the short exact sequence of holomorphic vector bundles
\begin{equation}\label{n2}
0\, \longrightarrow\,R^{d-1}q_*\Omega^d_{\mathcal Z/M}\, \longrightarrow\, R^{2d-1}q_*\underline{\C}
\, \longrightarrow\, R^dq_*\Omega^{d-1}_{\mathcal Z/M}\, \longrightarrow\, 0\, ,
\end{equation}
on $M$.

The holomorphic vector bundle $R^{2d-1}q_*\underline{\mathbb C}$ is equipped with the Gauss--Manin connection,
which is an integrable connection. The quotient $R^dq_*\Omega^{d-1}_{\mathcal Z/M}$, in \eqref{n2},
of $R^{2d-1}q_*\underline{\mathbb C}$ is a holomorphic vector
bundle on $M$ with fiber $H^d(L',\,\Omega^{d-1}_{L'})$ over any $L'\, \in\, M$.

The homomorphism
\begin{equation}\label{ep4}
(R^1 q_*\underline{\mathbb C})\otimes (R^{2d-1} q_*\underline{\mathbb C})\, \longrightarrow\,
\underline{\mathbb C}\, ,
\end{equation}
that sends any $\alpha\otimes \beta \, \in\,
(R^1 q_*\underline{\mathbb C})_t\otimes (R^{2d-1} q_*\underline{\mathbb C})_t$, $t\, \in\, M$,
to
$$
\int_{q^{-1}(t)} v\wedge w \in\, \mathbb C
$$
is covariant constant with respect to the connection on $(R^1 q_*\underline{\mathbb C})\otimes
(R^{2d-1} q_*\underline{\mathbb C})$ induced by the Gauss--Manin connections on
$R^1 q_*\underline{\mathbb C}$ and $R^{2d-1} q_*\underline{\mathbb C}$. Consequently,
the pairing in \eqref{ep4} yields a holomorphic isomorphism of vector bundles
$$
(q_*\Omega^1_{{\mathcal Z}/M})^* \,\stackrel{\sim}{\longrightarrow} \, R^dq_*\Omega^{d-1}_{\mathcal Z/M}
$$
over $M$. Combining this isomorphism with the isomorphism in \eqref{ep2}
we get a holomorphic isomorphism of vector bundles
\begin{equation}\label{e4'}
\widehat{\chi}\, :\, \Omega^1_M \,\stackrel{\sim}{\longrightarrow} \, R^dq_*\Omega^{d-1}_{\mathcal Z/M}
\end{equation}
over $M$.

We shall denote the total space of the holomorphic vector bundle
$R^dq_*\Omega^{d-1}_{\mathcal Z/M}$ by $\mathcal{W}$, so
\begin{equation}\label{e11}
\mathcal{W}\, :=\, R^dq_*\Omega^{d-1}_{\mathcal Z/M}\, \stackrel{\widehat\gamma}{\longrightarrow}\, M
\end{equation}
is a holomorphic fiber bundle.

As in Section \ref{sect2}, consider the canonical Liouville
holomorphic $1$-form on the total space of $\Omega^1_M$. Using the isomorphism in \eqref{e4'},
this Liouville $1$-form on the total space of $\Omega^1_M$ produces a
holomorphic $1$-form on $\mathcal{W}$ in \eqref{e11}. Let
$$
\theta'\, \in\, H^0(\mathcal{W}, \, \Omega^1_{\mathcal{W}})
$$
be this holomorphic $1$-form on $\mathcal{W}$. We note that
\begin{equation}\label{e8}
d\theta'\, \in\, H^0(\mathcal{W}, \, \Omega^2_{\mathcal{W}})
\end{equation}
is a holomorphic symplectic form on $\mathcal{W}$. Indeed, the isomorphism in
\eqref{e4'} takes $d\theta'$ to the Liouville symplectic form on the
total space of $\Omega^1_M$.

\begin{remark}\label{rem0}
Note that while the K\"ahler form $\omega$ on $X$ was used in the construction of the
symplectic form $d\theta$ on $\mathcal Y$ in \eqref{e10} (see the pairing in \eqref{ep}), the construction of
$d\theta'$ in \eqref{e8} does not use the K\"ahler form $\omega$ on $X$. We recall that the isomorphism
in \eqref{ep2} is constructed from the the pointwise isomorphisms in \eqref{n1}. Note that the
isomorphism in \eqref{n1} does not depend on the K\"ahler form $\omega$.
\end{remark}

The Albanese $\text{Alb}(Y)$ of a compact K\"ahler manifold $Y$ is defined to be
$$
\text{Alb}(Y)\,=\, H^0(Y,\,\Omega^1_Y)^*/H_1\,(Y,\,\Z)\,=\, H^n(Y,\,\Omega^{n-1}_Y)/H_1\,(Y,\,\Z)\, ,
$$
where $n\,=\, \dim_{\mathbb C} Y$ (see \cite[p.~331]{GH}). It is a compact
complex torus.

For each point $L\,\in\, M$, consider the composition of homomorphisms
$$
H^{2d-1}(L,\,\Z)\,\longrightarrow\, H^{2d-1}(L,\,{\mathbb C})\,\longrightarrow\,
H^{2d-1}(L,\,{\mathbb C})/H^{d-1}(L,\,\Omega_L^d)\,=\, H^d(L,\,\Omega_L^{d-1})
$$
(see \eqref{n2}). It produces a homomorphism
\begin{equation}\label{e14}
R^{2d-1}q_*\underline{\Z}\, \longrightarrow\,\mathcal{W}\, ,
\end{equation}
where $\mathcal{W}$ is defined in \eqref{e11}.

\begin{remark}\label{r-1}
The Gauss--Manin connection on $R^{2d-1}q_*\underline{\mathbb C}\,\longrightarrow\, M$ evidently preserves
the subbundle of lattices $$R^{2d-1}q_*\underline{\mathbb Z}\,\subset\,R^{2d-1}q_*\underline{\mathbb C}\, .$$
From this it follows immediately that the $C^\infty$ submanifold $R^{2d-1}q_*\underline{\mathbb Z}\,\subset\,
{\mathcal W}$ in \eqref{e14} is in fact a complex submanifold.
\end{remark}

The quotient
\begin{equation}\label{e9}
\widehat{\mathcal{A}}\,:=\,\mathcal{W}/(R^{2d-1}q_*\underline{\Z})\,\longrightarrow\, M
\end{equation}
for the homomorphism in \eqref{e14}
is in fact a holomorphic family of compact complex tori over $M$. Note that the fiber
of $\widehat{\mathcal{A}}$ over each $L\,\in\, M$
is the Albanese torus $$\mathrm{Alb}(L)\,=\,H^d(L,\,\Omega_L^{d-1})/H^{2d-1}(L,\,\Z)\, .$$

For any complex Lagrangian $L\, \in\, M$, using Serre duality,
$$
H^d(L,\,\Omega_L^{d-1})\,=\, H^0(L,\,\Omega_L^1)^*\, ,
$$
and the underlying real vector space for $H^0(L,\,\Omega_L^1)^*$ is identified with
$$H^1(L,\, {\mathbb R})^*\,=\, H_1(L,\, {\mathbb R})\, .$$ Using Poincar\'e duality for $L$, we have
$H^{2d-1}(L,\,\Z)/\text{Torsion}\,=\, H_1(L,\,\Z)/\text{Torsion}$. Consequently,
$\widehat{\mathcal{A}}$ in \eqref{e9} admits the following isomorphism:
\begin{equation}\label{e12}
\widehat{\mathcal{A}}\, =\,\mathcal{W}/(R^{2d-1}q_*\underline{\Z})\,=\,
\mathcal{W}/\widetilde{H}_1(\Z)\,=\,
(q_*\Omega^1_{\mathcal Z/M})^*/\widetilde{H}_1(\Z)\, ,
\end{equation}
where $\widetilde{H}_1(\Z)$ is the local system on $M$ whose stalk over any $L\, \in\, M$ is
$H_1(L,\,\Z)/\text{Torsion}$. Also we have the isomorphism of real tori
\begin{equation}\label{e21}
\widehat{\mathcal{A}}\, =\, (R^1q_*\underline{\mathbb R})^*/\widetilde{H}_1(\Z)\,=\,
\widetilde{H}_1({\mathbb R})/\widetilde{H}_1(\Z)\, ,
\end{equation}
where $\widetilde{H}_1({\mathbb R})$ is the local system on $M$ whose stalk over any $L\, \in\, M$ is
$H_1(L,\,{\mathbb R})$.

\begin{theorem}\label{theo1}
The $2$-form $d\theta'$ in \eqref{e8} on ${\mathcal W}$ descends to the quotient torus
$\widehat{\mathcal A}$ in \eqref{e9}.
\end{theorem}

\begin{proof}
Take a point $L_0\,\in\, M$. In \cite{Hi2} Hitchin constructed a $C^\infty$ coordinate function
on $M$ defined around the point $L_0\,\in\, M$
that takes values in $H^1(L_0,\, {\mathbb R})$ \cite[p.~79, Theorem 2]{Hi2}; we will briefly recall
this construction.

Take any
\begin{equation}\label{ec}
c\,\in\, H_1(L_0,\, {\mathbb Z})/{\rm Torsion}\, .
\end{equation}
Let $U\,\subset\, M$
be a contractible neighborhood of the point $L_0$. Choose a $S^1$-subbundle of
the fiber bundle $\mathcal Z$ (see \eqref{e-1}) over $U$
\begin{equation}\label{fb}
B\,\stackrel{\iota_U}{\hookrightarrow}\, {\mathcal Z}\vert_U \, \stackrel{q}{\longrightarrow}\, U
\end{equation}
such that the fiber of the $S^1$-bundle $B$ over $L_0$ represents the homology
class $c$ in \eqref{ec}; recall that the fiber of $\mathcal Z$ over the point $L_0\,\in\,
M$ is the Lagrangian $L_0$ itself.

Consider the symplectic form $\Phi$ on $X$. Integrating $\iota^*_Up^*
{\rm Re}(\Phi)$ along the fibers of $B$, where $\iota_U$ and $p$ are the maps in \eqref{fb}
and \eqref{pq} respectively, we get a closed $1$-form $\xi_c$ on $U$. Let $f_c$ be the
unique function on $U$ such that $f_c(L_0)\,=\, 0$ and $df_c\,=\, \xi_c$. Now, let
$$
\mu\, :\, U \, \longrightarrow\, H^1(L_0,\, {\mathbb R})
$$
be the function uniquely determined by the condition that $\phi_{L_x}(c\otimes \mu(x)) \,=\, f_c(x)$ for all
$x\,\in\, U$ and $c\,\in\, H_1(L_0,\, {\mathbb Z})/{\rm Torsion}$, where
$\phi_{L_x}$ is the pairing constructed as in \eqref{e2} for the complex Lagrangian $L_x\,=\, q^{-1}(x)
\, \subset\, X$, where $q$ is the projection
in \eqref{pq}. This $\mu$ is a local diffeomorphism \cite[p.~79, Theorem 2]{Hi2}.

Using the K\"ahler form $\omega_{L_0}\,:=\, \omega\vert_{L_0}$ on $L_0$ (see \eqref{wl}), we identify
$H^1(L_0,\, {\mathbb R})$ with $H_1(L_0,\, {\mathbb R})$ as follows. Since the pairing
$$
H^1(L_0,\, {\mathbb R})\otimes H^1(L_0,\, {\mathbb R})\,\longrightarrow {\mathbb R}\, , \ \
v\otimes w\, \longmapsto\, \int_L v\wedge w\wedge\omega^{d-1}_{L_0}
$$
is nondegenerate, it produces an isomorphism
\begin{equation}\label{ei}
H^1(L_0,\, {\mathbb R})\,\stackrel{\sim}{\longrightarrow}\,
H^1(L_0,\, {\mathbb R})^* \,=\, H_1(L_0,\, {\mathbb R})\, .
\end{equation}
On the other hand, there is the natural homomorphism $H^1(L_0,\, {\mathbb Z})\, \longrightarrow\,
H^1(L_0,\, {\mathbb R})$. Let
\begin{equation}\label{ga}
\Gamma\, \subset\, H_1(L_0,\, {\mathbb R})
\end{equation}
be the subgroup that corresponds to
$H^1(L_0,\, {\mathbb Z})$ by the isomorphism in \eqref{ei}.

Using the above coordinate function $\mu$ on $U$, we have
$$
{\mathcal T}^*U\,\stackrel{\sim}{\longrightarrow}\, {\mathcal T}^*\mu(U) \,=\,\mu(U)\times
H^1(L_0,\, {\mathbb R})^*\,=\, \mu(U)\times H_1(L_0,\, {\mathbb R})\, ,
$$
where ${\mathcal T}^*$ denotes the real cotangent bundle.

The Liouville symplectic form on ${\mathcal T}^*\mu(U)$ is clearly the constant $2$-form 
on $$H^1(L_0,\, {\mathbb R})\times H_1(L_0,\, {\mathbb R})$$ given by the
natural isomorphism of $H_1(L_0,\, {\mathbb R})$ with $H^1(L_0,\, {\mathbb R})^*$.
From this it follows immediately that
$$
\mu(U)\times \Gamma\, \subset\, \mu(U)\times H_1(L_0,\, {\mathbb R})
$$
is a Lagrangian submanifold with respect to the Liouville symplectic form on ${\mathcal T}^*\mu(U)$,
where $\Gamma$ defined in \eqref{ga}.

Since $\mu(U)\times \Gamma\, \subset\, \mu(U)\times H_1(L_0,\, {\mathbb R})$ is a Lagrangian submanifold,
it follows that for ${\mathcal W}\,=\,
(R^1q_*\underline{\mathbb R})^*$ (see \eqref{e21} and \eqref{e11}), the image of the natural map
$$
\widetilde{H}_1(\Z) \, \longrightarrow\, (R^1q_*\underline{\mathbb R})^*\,=\, {\mathcal W}
$$
(see \eqref{e12}) is Lagrangian with respect to the real symplectic form $\text{Re}(d\theta')$ on
$\mathcal W$, where $d\theta'$ is constructed in \eqref{e8}.

It was noted in Remark \ref{r-1} that $R^{2d-1}q_*\underline{\mathbb Z}$ is a complex submanifold of ${\mathcal W}$. 
Consequently, the $2$-form on $R^{2d-1}q_*\underline{\mathbb Z}$ obtained by restricting the 
holomorphic $2$-form $d\theta'$ on $\mathcal W$ is also holomorphic. Since the real part of the holomorphic 
$2$-from on $R^{2d-1}q_*\underline{\mathbb Z}$ given by $d\theta'$ vanishes identically, we conclude that the
holomorphic $2$-from on $R^{2d-1}q_*\underline{\mathbb Z}$, given by $d\theta'$, itself vanishes
identically. Therefore, $R^{2d-1}q_*\underline{\mathbb Z}$ is a Lagrangian submanifold of the holomorphic symplectic 
manifold ${\mathcal W}$ equipped with the holomorphic symplectic form $d\theta'$.

To complete the proof we recall a general property of the Liouville symplectic form.

Let $N$ be a manifold and $\alpha$ a $1$-form on $N$. Let
$$
t\, :\, T^*N \,\longrightarrow\, T^*N
$$
be the diffeomorphism that sends any $v \,\in\, T^*_nN$ to $v+\alpha(n)$. If $\psi$ is the
Liouville symplectic form on $T^*N$, the
$$
t^*\psi \,=\, \psi+d\alpha \, .
$$
In particular, the map $t$ preserves $\psi$ if and only if the form $\alpha$ is closed.
Also, the image of $t$ is Lagrangian submanifold of $T^*N$ for $\psi$ if and only
$\alpha$ is closed.

Since $$R^{2d-1}q_*\underline{\mathbb Z}$$ is a Lagrangian submanifold of the symplectic manifold
$({\mathcal W},\, d\theta')$, from the above property of the Liouville symplectic form it
follows immediately that the $2$-form $d\theta'$ on ${\mathcal W}$ descends to the quotient space
$\widehat{\mathcal A}$ in \eqref{e9}.
\end{proof}

Let
\begin{equation}\label{qa}
q_{\widehat{\mathcal A}}\, :\, {\mathcal W}\, \longrightarrow\, \widehat{\mathcal A}\, :=
\,{\mathcal W}/(R^{2d-1}q_*\underline{\Z})
\end{equation}
be the quotient map (see \eqref{e9}).
From Theorem \ref{theo1} we know that there is a unique $2$-form
\begin{equation}\label{ta}
\Theta_{\widehat{\mathcal A}}\, \in\, H^0(\widehat{\mathcal A},\, \Omega^2_{\widehat{\mathcal A}})
\end{equation}
such that
\begin{equation}\label{ta2}
q^*_{\widehat{\mathcal A}}\Theta_{\widehat{\mathcal A}}\,=\, d\theta'\, ,
\end{equation}
where $q_{\mathcal A}$ is the map
in \eqref{qa}. Since $d\theta'$ is a holomorphic symplectic form, it follows immediately that
$\Theta_{\widehat{\mathcal A}}$ is a holomorphic symplectic form on $\widehat{\mathcal A}$.

The projection $\widehat{\gamma}\, :\, {\mathcal W}\,\longrightarrow \, M$ in \eqref{e11} clearly
descends to a map from ${\mathcal A}$ to $M$. Let
\begin{equation}\label{vp}
\varpi\, :\, \widehat{\mathcal{A}}\, \longrightarrow\, M
\end{equation}
be the map given by $\widehat{\gamma}$; so we have
$$
\widehat{\gamma}\,=\, \varpi\circ q_{\widehat{\mathcal A}}\, ,
$$
where $q_{\widehat{\mathcal A}}$ is constructed in \eqref{qa}.

\begin{lemma}\label{lem-1}
The projection $\varpi$ in \eqref{vp} defines a
completely integrable structure on $\widehat{\mathcal{A}}$ for the symplectic
form $\Theta_{\widehat{\mathcal A}}$ constructed in \eqref{ta}. In particular, the
fibers of $\varpi$ are Lagrangians with respect to the
symplectic form $\Theta_{\widehat{\mathcal A}}$.
\end{lemma}

\begin{proof}
Recall that $\mathcal W$ is holomorphically identified with $\Omega^1_M$ by the map
$\widehat{\chi}$ in \eqref{e4'} (see \eqref{e11}).
This map $\widehat\chi$ takes the Liouville symplectic form on $\Omega^1_M$ to the
symplectic form $d\theta'$ on $\mathcal W$. Therefore, from \eqref{qa} and \eqref{ta2} we conclude that
$\widehat{\mathcal A}$ is locally isomorphic to $\Omega^1_M$ such that the projection $\varpi$ is taken to the
natural projection $\Omega^1_M\, \longrightarrow\, M$, and the symplectic form $\Theta_{\widehat{\mathcal A}}$
on $\widehat{\mathcal A}$ is taken to the Liouville symplectic form on $\Omega^1_M$. The lemma follows
immediately from these, because the natural projection $\Omega^1_M\, \longrightarrow\, M$ defines a
completely integrable structure on $\Omega^1_M$ for the Liouville symplectic form.
\end{proof}

\section{The relative Picard group}\label{se4}

For any $L\,\in\, M$, consider the homomorphisms
\begin{equation}\label{h1}
H^1(L,\, {\mathbb Z}) \,\longrightarrow \, H^1(L,\, {\mathbb C})\,\longrightarrow \, H^1(L,
\, {\mathcal O}_L)\, ,
\end{equation}
where $H^1(L,\, {\mathbb Z}) \,\longrightarrow \, H^1(L,\, {\mathbb C})$ is the natural
homomorphism given by the inclusion of $\mathbb Z$ in $\mathbb C$, and the projection
$H^1(L,\, {\mathbb C})\,\longrightarrow \, H^1(L,\, {\mathcal O}_L)$ corresponds to
the isomorphism $$H^1(L,\, {\mathbb C})/H^0(L,\, \Omega^1_L)\,=\,
H^1(L,\, {\mathcal O}_L)$$ (see \eqref{a1}).
The image of the composition of homomorphisms in \eqref{h1} is actually a cocompact lattice in $H^1(L, 
\, {\mathcal O}_L)$; so $H^1(L,\, {\mathcal O}_L)/H^1(L,\, {\mathbb Z})$ is a compact complex torus.
We note that the composition of homomorphisms in \eqref{h1} is in fact injective. When the compact
complex manifold $L$ is a complex projective variety, then
$H^1(L,\, {\mathcal O}_L)/H^1(L,\, {\mathbb Z})$ is in fact an abelian variety.

As $L$ moves over the family $M$, these cocompact lattices fit together to produce
a $C^\infty$ submanifold of the complex manifold ${\mathcal Y}$ in \eqref{e6}.

The Gauss--Manin connection on $R^1q_*\underline{\mathbb C}\,\longrightarrow\, M$ evidently preserves
the above bundle of cocompact lattices $R^1q_*\underline{\mathbb Z}\,\subset\,R^1q_*\underline{\mathbb C}$.
From this it follows immediately that the above $C^\infty$ submanifold $R^1q_*\underline{\mathbb Z}\,\subset\,
{\mathcal Y}$ is in fact a complex submanifold.

Taking fiber-wise quotients, we conclude that
\begin{equation}\label{e7}
{\mathcal A}\, := \,{\mathcal Y}/(R^1q_*\underline{\mathbb Z})\, \longrightarrow\, M
\end{equation}
is a holomorphic family of compact complex tori over $M$.

\begin{remark}\label{rem1}
Let $Y$ be a compact K\"ahler manifold. Consider the short exact sequence of sheaves on $Y$ given by the
exponential map
$$
0\, \longrightarrow\, \underline{\mathbb Z} \, \longrightarrow\, {\mathcal O}_Y
\, \stackrel{\lambda\mapsto \exp(2\pi\sqrt{-1}\lambda)}{\longrightarrow}\, {\mathcal O}^*_Y
\, \longrightarrow\, 0\, ,
$$
where ${\mathcal O}^*_Y$ is a multiplicative sheaf of holomorphic functions with values in
${\mathbb C}\setminus\{0\}$. For the corresponding long exact sequence of cohomologies
$$
H^1(Y,\, \underline{\mathbb Z}) \, \longrightarrow\,
H^1(Y,\, {\mathcal O}_Y) \, \longrightarrow\, H^1(Y,\, {\mathcal O}^*_Y) \, \stackrel{c_1}{\longrightarrow}\,
H^2(Y,\, \underline{\mathbb Z})\, ,
$$
where the connecting homomorphism $c_1$ sends any holomorphic line bundle
$\xi\, \in\, H^1(Y,\, {\mathcal O}^*_Y)$ to $c_1(\xi)$, the quotient
$$
H^1(Y,\, {\mathcal O}_Y)/H^1(Y,\, \underline{\mathbb Z})
$$
gets identified with the Picard group $\text{Pic}^0(Y)$ that parametrizes the topologically trivial
holomorphic line bundles on $Y$.

Consequently, the quotient $\mathcal{A}$ in \eqref{e7}
is naturally identified with the moduli space
${\rm Pic}^0_{{\mathcal Z}/M}$ of topologically trivial holomorphic line bundles on the
fibers of $q\, :\, {\mathcal Z}\, \longrightarrow\, M$. In other words, $\mathcal{A}$
parametrizes all pairs of the form $(L,\, \xi)$, where $L\,\in\, M$, and $\xi$ is a
topologically trivial holomorphic line bundle on $L$.
\end{remark}

Recall that we have the holomorphic symplectic form $d\theta$ on ${\mathcal
Y}$, where $\theta$ is constructed in \eqref{e5}.

\begin{remark}
Recall from Remark \ref{rem0} that $d\theta$ does depend on the K\"ahler form
$\omega$ on $X$.
\end{remark}

\begin{proposition}\label{proposi1}
The $2$-form $d\theta$ on ${\mathcal Y}$ descends to the quotient space
${\mathcal A}$ in \eqref{e7}, or in other words, $d\theta$ is the pullback of a $2$-form
on $\mathcal A$.
\end{proposition}

\begin{proof}
The proof of the proposition is very similar to the proof of Theorem \ref{theo1}.
As before, the local coordinate functions on $M$ constructed in \cite{Hi2} play a crucial
role. We omit the details of the proof.
\end{proof}

Let
\begin{equation}\label{qa2}
q_{\mathcal A}\, :\, {\mathcal Y}\, \longrightarrow\, {\mathcal A}\, :=
\,{\mathcal Y}/(R^1q_*\underline{\mathbb Z})
\end{equation}
be the quotient map (see \eqref{e7}). Let
$\Theta_{\mathcal A}\, \in\, H^0({\mathcal A},\, \Omega^2_{\mathcal A})$ be the
holomorphic symplectic form given by Proposition \ref{proposi1}, so
\begin{equation}\label{ta22}
q^*_{\mathcal A}\Theta_{\mathcal A}\,=\, d\theta\, ,
\end{equation}
where $q_{\mathcal A}$ is the map
in \eqref{qa2}.

The projection $\gamma$ in \eqref{e6} clearly
descends to a map
\begin{equation}\label{vp2}
\varpi_0\, :\, \mathcal{A}\, \longrightarrow\, M\, .
\end{equation}

Note that the isomorphism between $\mathcal{A}$ and ${\rm Pic}^0_{{\mathcal Z}/M}$ in
Remark \ref{rem1} takes $\varpi_0$ to the forgetful map
$$
{\rm Pic}^0_{{\mathcal Z}/M}\, \longrightarrow\, M
$$
that forgets the line bundle, or in other words, it sends any $(L,\, \xi)\,\in\,
{\rm Pic}^0_{{\mathcal Z}/M}$ to the complex Lagrangian $L$ forgetting the line bundle
$\xi$.

\begin{lemma}\label{lem2}
The projection $\varpi_0$ in \eqref{vp2} defines a
completely integrable structure on $\mathcal{A}$ for the symplectic
form $\Theta_{\mathcal A}$.
\end{lemma}

\begin{proof}
Recall that $\mathcal Y$ is holomorphically identified with $\Omega^1_M$ by the map
$\chi$ in \eqref{e4}, and $\chi$ takes the Liouville symplectic form on $\Omega^1_M$ to the
symplectic form $d\theta$ on $\mathcal Y$. Therefore, from \eqref{qa} and \eqref{ta22} we conclude that
$\mathcal A$ is locally isomorphic to $\Omega^1_M$ such that the projection $\varpi_0$ is taken to the
natural projection $\Omega^1_M\, \longrightarrow\, M$, and the symplectic form $\Theta_{\mathcal A}$
on $\mathcal A$ is taken to the Liouville symplectic form on $\Omega^1_M$. The lemma follows
immediately from these.
\end{proof}

\section*{Acknowledgements}

We thank the referee for helpful comments.
The first author is partially supported by a J. C. Bose fellowship.
The second author is supported by 
Ministerio de Ciencia e Innovaci\'on of Spain
(grants MTM2016-79400-P, PID2019-108936GB-C21, and 
ICMAT Severo Ochoa project SEV-2015-0554) and CSIC
(2019AEP151 and \textit{Ayuda extraordinaria a Centros de Excelencia 
Severo Ochoa} 20205CEX001) The third author is supported by CMUP, financed by national funds
through FCT -- Funda\c{c}\~ao para a Ci\^encia e a Tecnologia, I.P., under the project with
reference UIDB/00144/2020. The second and third authors wish to thank 
TIFR Mumbai for hospitality and for the excellent conditions provided.

%%%%%%%%%%%%%%%%%%%%%%%%%%%%%%%%%%%%%%%%%%%%%%%%%%%%%%%%%%%%%%

\end{document}